\documentclass[letterpaper, 10 pt, conference]{ieeeconf}  

\IEEEoverridecommandlockouts                              

\usepackage{amsmath} 

\usepackage{amssymb}  
\usepackage{optidef}

\usepackage{graphicx}
\usepackage{amsmath}
\usepackage{amssymb}
\usepackage{amsthm}
\usepackage{dsfont}
\usepackage[noadjust]{cite}

\usepackage[american]{circuitikz}
\usepackage{tikz}

\newtheorem{thm}{Theorem}
\newtheorem{lem}{Lemma}

\theoremstyle{definition}
\newtheorem{defn}{Definition}

\newtheorem*{remark}{Remark}

\newcommand{\X}{\bold{\mathcal{X}}}
\newcommand{\Sx}{S_x}
\newcommand{\Su}{S_u}
\newcommand{\Uo}{\bold{\mathcal{U}}}
\newcommand{\R}{\mathbb{R}}

\newcommand{\mydot}{\boldsymbol{\cdot}}
\title{\LARGE \bf
Adversarial Model Predictive Control via Second-Order Cone Programming
}

\author{James Guthrie and Enrique Mallada
\thanks{$^{1}$James Guthrie and Enrique Mallada are with the Johns Hopkins University, Baltimore, Maryland, USA.
        {\tt\small jguthri6@jhu.edu, mallada@jhu.edu}}%
\thanks{The work was supported by ARO through contract W911NF-17-1-0092, US DoE EERE award DE-EE0008006, and NSF through grants CNS 1544771, EPCN 1711188, AMPS 1736448, and CAREER 1752362.}
}

\begin{document}

\maketitle
\thispagestyle{empty}
\pagestyle{empty}

\begin{abstract}
We study the problem of designing attacks to safety-critical systems in which the adversary seeks to maximize the overall system cost within a model predictive control framework. Although in general this problem is NP-hard, we characterize a family of problems that can be solved in polynomial time via a second-order cone programming relaxation. In particular, we show that positive systems fall under this family. We provide examples demonstrating the design of optimal attacks on an autonomous vehicle and a microgrid.
\end{abstract}

\section{INTRODUCTION}
Safety-critical systems increasingly rely on distributed feedback for their underlying control algorithms. In these cyber-physical systems, the action of individual agents is impacted by the state of other agents which is either sensed directly or obtained over communication channels. Common examples include, for instance, power grids and vehicle platoons. Given the critical nature of these systems, it is essential to ensure that the control algorithms utilized are robust to adversarial attacks which can take many forms. 

For example, in false data injection attacks, an adversary takes control over communication channels and corrupts the feedback data to compromise the system performance. Much recent work has focused on designing and detecting false data injection attacks within power systems \cite{Liu2009, Pas2013, Jin2017}. Alternatively, instead of corrupting feedback channel information, an attacker could compromise existing agents or introduce new adversarial agents with the aim of degrading system performance. Examples include adding a rogue car to a vehicle platoon \cite{Deb2015} or malicious demand response in power grids \cite{Bro2018}. Lastly, instead of injecting false sensor data or introducing adversarial agents, the attacker might take over the whole system and control it with an antagonistic algorithm \cite{Lip2016} that maximizes damage. 

Performance of these cyber-physical systems is often measured with respect to a convex quadratic cost function. For example, consensus problems seek to minimize the disagreement between agents. Regulation problems seek to minimize the deviation from a desired equilibrium condition. In designing attacks on these systems, it is therefore natural to seek to maximize these same objectives. This leads to a non-convex problem which is NP-hard in general. Due to the computational complexity, suboptimal solutions are typically sought via convex-concave approximations \cite{Lip2016}, semidefinite relaxations \cite{Jin2017}, or general nonlinear programming methods. Alternatively, an attacker may avoid the non-convex problem by selecting a target state (which is different from the system's intended operational state) and minimizing deviations from it \cite{Che2018}. While the resulting problem is convex, the choice of target state is arbitrary and up to the attacker to determine. Thus the target state often acts as a surrogate for true adversarial intentions.       
      
This paper seeks a different approach. Instead of looking for suboptimal or surrogate solutions, we focus on instances in which the non-convex problem can be solved to global optimality. By leveraging optimality guarantees for second-order cone program (SOCP) relaxations of non-convex quadratically-constrained quadratic programs (QCQPs), we provide a characterization of a family of systems that are highly susceptible to adversarial attacks. Surprisingly, the characterized family includes, as a special case, positive systems with non-positive quadratic objectives and constraints. 

This has application to many cyber-physical systems, including micro-grids \cite{Tei2015} and vehicle platoons \cite{Ran2011} which often exhibit positive dynamics. Our results suggest that these systems are highly vulnerable to adversarial attacks and promotes the need of further research into the development of new methodologies that can make these systems less vulnerable to such attacks.

The rest of the paper is organized as follows. Section II introduces some preliminaries, including the formal definition of QCQP, an overview of the SOCP relaxation used in this paper, and the definition of positive systems. Section III formalizes the adversarial MPC problem to be used in this paper, as well some useful reformulations. Section IV establishes conditions under which a non-convex MPC problem has an exact SOCP relaxation. Section V provides a few numerical illustrations of our approach, and Section VI concludes the paper and discusses future directions.

\newcommand{\Zero}{{0}}
\subsection{Notation}
Let $\mathbb{S}^n$ denote the set of $n \times n$ symmetric matrices, $\mathbb{N}$ denote the set of non-negative integers, $\mathbb{N}^+$ the set of positive integers, and $A^{T}$ denote the transpose of a matrix $A$. Let $a_j$ denote the element $j$ of vector $a \in \R^n$ and $[A]_{jk}$ denote element $(j,k)$ of matrix $A$. The inequalities $\leq$, $\geq$ are to be interpreted element-wise.
$I_n$ denotes the $n \times n $ identity matrix, $\Zero_{m \times n}$ the $m \times n$ zero matrix, and $\mathds{1}_n$ a vector in $\R^n$ with all entries equal to $1$. We occasionally drop subscripts where dimensions can be inferred from context. 
For $A$, $B \in \mathbb{S}^n$, let $A \mydot B = \sum_{j=1}^{n}\sum_{k=1}^{n}[A]_{jk}[B]_{jk}$.

\section{PRELIMINARIES}
\subsection{Exact Solutions of Some Non-Convex QCQPs}
We first review the main result of \cite{Kim2003} regarding the conditions under which non-convex QCQPs can be solved exactly via a SOCP relaxation. Consider the following QCQP
\begin{mini}|s|
{{z}}{{z}^{T}{Q}_0{z}+2{q}_0^{T}{z} + \gamma_0}
{}{}
\addConstraint{{z}^{T}{Q}_i{z}+2{q}_i^{T}{z} + \gamma_i \leq 0 }{,\quad}{i = 1,\hdots,m}
\label{eq:QCQP} 
\end{mini} 
\par\noindent
where $z \in \R^n$, $Q_p \in \mathbb{S}^n$, $q_p \in \R^n$, $\gamma_p \in \R$ and $p \in \{0,1,...,m\}$. Define the following matrix: 
\begin{equation} \label{eq:M_def}
{P}_p = \begin{bmatrix}
\gamma_p && {q}_p^{T} \\
{q}_p && {Q}_p
\end{bmatrix} 
\end{equation}
We rewrite the QCQP in homogeneous form as:
\newcommand{\z}{\begin{bmatrix}
1 \\
{z}
\end{bmatrix}}
\begin{mini}|s|[0]
{{z}}
{\z^{T}{P}_0 \z}{}{}
{\label{homQCQP}}
\addConstraint{\z^{T}{P}_i\z}{\leq 0,}{\quad i = 1,\hdots,m}
\end{mini}
where $z \in \R^n$, and $P_p \in \mathbb{S}^{n+1}$ for $p \in \{0,1,...,m\}$. 

Herein we make no assumptions about the sign definiteness of matrices $P_p$. When $P_0$ contains at least one negative eigenvalue, problem (\ref{homQCQP}) is non-convex and NP-hard to solve \cite{Par1991}. In \cite{Kim2003} it was shown that if the matrices collectively satisfy a specific sign property (defined below), the non-convex QCQP can be solved to global optimality via a second-order cone program. For convenience, we restate the relevant definitions and theorem of \cite{Kim2003}.

\begin{defn}[\cite{Kim2003}]  \label{def:odnp}
A symmetric matrix $A \in  \mathbb{S}^n $ is said to be \textit{almost off-diagonal non-positive} if 
there exists a sign vector $\bold{\sigma} \in \{-1,+1\}^n$  such that
$[A]_{jk}\sigma_j \sigma_k \leq 0, \ (0 \leq j < k \leq n) $
\end{defn}
\begin{defn}[\cite{Kim2003}] \label{def:odnp_family}
A family of symmetric matrices $A_p \in \mathbb{S}^n \ (0 \leq p \leq m)$ is said to be \textit{uniformly almost off-diagonal non-positive}
if there exists a sign vector $\sigma \in \{-1,+1\}^n $ such that $[A_p]_{jk}\sigma_j \sigma_k \leq 0, \ (1 \leq j < k \leq n, \ 0 \leq p \leq m) $
\end{defn}  

\newcommand{\zt}{\begin{bmatrix}
z_0^{T} && \bold{z}^{T}
\end{bmatrix}}

\newcommand{\ztx}{\begin{bmatrix}
\sigma_0 \sigma_1 \sqrt{[X]_{11}} && \hdots && \sigma_0 \sigma_n \sqrt{[X]_{nn}}
\end{bmatrix}}

\newcommand{\norm}[1]{\left\lVert#1\right\rVert}
\newcommand{\normentry} { \begin{bmatrix} [X]_{jj} - [X]_{kk} 
\\ 2[X]_{jk}
\end{bmatrix}}
\begin{thm}[\cite{Kim2003}] \label{thmSOCP}
Consider a QCQP of the form (\ref{homQCQP}) in which the family of symmetric
matrices $P_p \in \mathbb{S}^{n+1} ,\ (0 \leq p \leq m) $
is uniformly almost off-diagonal non-positive with respect to a sign vector $\sigma \in \{-1,+1\}^{n+1}$. Let $\Lambda = \{(j,k): [P_p]_{jk} \neq 0 $ for some $0 \leq p \leq m, \ 0 \leq j < k \leq n\}$.
Then
\begin{equation}
z = \ztx^T
\label{eq:z}
\end{equation}
is an optimal solution of (\ref{homQCQP}) where $X$ is the optimal solution of the following second-order cone program:
\begin{equation}
\begin{aligned}
\min_{X} &\quad {P}_0 \mydot {X}\\
\emph{\text{s.t.}} &\quad P_i \mydot X \leq 0, \quad i = 1,\hdots,m,\\
                    &\quad [X]_{00} = 1,  \\
                    &\quad \norm{\normentry}_2 \leq [X]_{jj} + [X]_{kk} ,(j,k) \in \Lambda \\
\end{aligned}
\label{eq:SOCP}
\end{equation}
\end{thm}
\subsection{Positive Systems}
Consider a discrete-time linear system
\begin{equation} \label{eq:discrete_dynamics}
x(k+1) = Ax(k) + Bu(k)
\end{equation}
where $k \in \mathbb{N}, x(k) \in \mathbb{R}^{n_x}$ and $u(k) \in \mathbb{R}^{n_u} $. Let $x(0)$ denote the initial state of the system. 
\begin{defn} \label{def:positiveSystem}A discrete-time linear system is said to be \textit{positive} if $A \geq 0$ and $B \geq 0$. 
\end{defn}
\begin{lem}[\cite{Ran2018}] \label{lem:positiveSystem} Consider a positive system $(A,B)$ with initial condition $x(0) \geq 0$. Given an input sequence $u(k) \geq 0$, $(0 \leq k \leq N-1)$, then $x(k) \geq 0$, $(1 \leq k \leq N)$. 
\end{lem}
\section{PROBLEM SETUP}
\subsection{Model Predictive Control}

Consider a discrete-time, linear system model
\begin{equation} \label{eq:discrete_dynamics}
x(k+1) = Ax(k) + Bu(k)
\end{equation}
\par\noindent
where $ x(k) \in \mathbb{R}^{n_x}$ and $u(k) \in \mathbb{R}^{n_u} $. Let $x(0)$ denote the initial state of the system. To reduce notational clutter, we will write $x(0)$ as $x_0$ in the following. Standard linear MPC determines the optimal sequence of control actions over a prediction horizon $N \in \mathbb{N}^+$ to minimize a given quadratic cost function while respecting constraints on the system states and controls \cite{Bor2017}.
For convenience, define the following:
\begin{equation}
\begin{gathered}
\X = 
\begin{bmatrix} 
x(0)\\
x(1)\\
\vdots \\
x(N)
\end{bmatrix}
\Uo = 
\begin{bmatrix} 
u(0)\\
u(1)\\
\vdots \\
u(N-1)
\end{bmatrix}
\Sx = 
\begin{bmatrix} 
I\\
A\\
\vdots \\
A^N
\end{bmatrix} 
\\
\Su = 
\begin{bmatrix} 
0 && \hdots && \hdots && 0 \\
B && 0 && \hdots && 0 \\
AB && \ddots && \ddots && \vdots \\
\vdots && \ddots && \ddots && 0 \\
A^{N-1}B && \hdots && \hdots && B
\end{bmatrix} 
\end{gathered} \label{eq:XUSxSu}
\end{equation}

The system dynamics over the horizon $N$ then evolve according to
\begin{equation} \label{eq:dynCon}
\X = \Sx x_0 + \Su \Uo
\end{equation}

\newcommand{\XUo}{$\begin{bmatrix} \X \\ \Uo \end{bmatrix}$}
The MPC cost and constraint functions will be represented by generic quadratic functions of the form
\begin{multline} \label{eq:F}
F_i(\X,\Uo) = \\
\begin{bmatrix} \X \\ \Uo \end{bmatrix}^T \begin{bmatrix} Q_i && S_i \\ S_i^{T} && R_i \end{bmatrix} \begin{bmatrix} \X \\ \Uo \end{bmatrix} + 2
\begin{bmatrix} q_i \\ r_i \end{bmatrix}^T \begin{bmatrix} \X \\ \Uo  \end{bmatrix} + \gamma_i
\end{multline} 
\newline
where $\smash{Q_i \in \mathbb{S}^{(N+1)n_x}, R_i \in \mathbb{S}^{Nn_u}, S_i \in \mathbb{R}^{(N+1)n_x \times Nn_u}},$
\newline
$q_i \in \R^{(N+1)n_x}, r_i \in \R^{Nn_u}, \gamma_i \in \R$.
\newline The MPC problem can then be written compactly as:
\begin{mini}|s|[2] 
{\bold{\X, \Uo}}{F_0(\X,\Uo)}
{}{}
\addConstraint{\X}{= \Sx x_0 + \Su \Uo}
\addConstraint{F_i(\X,\Uo)}{\leq 0,}{\quad i = 1,\hdots,m}
\label{eq:MPC_Unc}
\end{mini} 
In this formulation, both the state $\X$ and control sequence $\Uo$ are decision variables.
\subsection{Condensed MPC}
We next project the quadratic cost and constraint functions $F_i(\X,\Uo)$ onto the dynamic equality constraint (\ref{eq:dynCon}), eliminating the state vector $\X$ as a decision variable. This is often referred to as the  condensed formulation as the resulting problem is of smaller dimension but with less sparsity in the matrices. We substitute (\ref{eq:dynCon}) for $\X$ in (\ref{eq:F}) and define a new quadratic function of the form
\begin{multline} \label{eq:Gdef1}
G_i(x_0,\Uo) = \Uo^T M_i \Uo + 2(x_0^TN_i + d_i^T)\Uo \\+ x_0^TT_ix_0 + 2v_i^Tx_0 + \gamma_i
\end{multline}
where
\begin{align} 
M_i& = S_u^TQ_iS_u + S_u^TS_i + S_i^TS_u + R_i \label{eq:Mi} \\
N_i& = S_x^TQ_iS_u + S_x^TS_i \label{eq:Ni} \\
d_i& = S_u^Tq_i + r_i \label{eq:di} \\
T_i& = S_x^TQ_iS_x \label{eq:Ti} \\
v_i& = S_x^Tq_i \label{eq:vi}
\end{align}
The condensed MPC formulation is then written as:
\begin{mini}|s| 
{\Uo}{G_0(x_0,\Uo)}
{}{}
\addConstraint{G_i(x_0,\Uo)}{\leq 0,}{\quad i = 1,\hdots,m}
\label{eq:MPC_Con}
\end{mini} 
Lastly, we rewrite the functions $G_i(x_0,\Uo)$ in homogeneous form by defining the following matrix: 
\begin{equation} \label{eq:Pix0}
P_i(x_0) = \begin{bmatrix} 
x_0^TT_ix_0 + 2v_i^Tx_0 + \gamma_i && (x_0^TN_i + d_i^T) \\
(N_i^Tx_0 + d_i) && M_i
\end{bmatrix}
\end{equation} 
\newcommand{\Uxx}{
\begin{bmatrix}
1 \\
\Uo 
\end{bmatrix}
}
We obtain the following equivalent homogeneous quadratic program:
\begin{mini}|s|[2] 
{\Uo}{\Uxx^TP_0(x_0)\Uxx}{}{}
\addConstraint{\Uxx^TP_i(x_0)\Uxx}{\leq 0,}{\quad i = 1,\hdots,m} \label{eq:MPC_Hom} 
\end{mini} 
\begin{remark}
Although the homogeneous quadratic form is a less common MPC formulation, it will allow us to readily apply the proposed SOCP relaxation of Theorem \ref{thmSOCP}.
\end{remark}

\section{Adversarial MPC with Non-Convex Quadratic Functions}
Provided the matrices $P_i~(i = 0,\hdots,m)$ of (\ref{eq:MPC_Hom}) satisfy the conditions of Theorem \ref{thmSOCP}, the (possibly non-convex) QCQP can be solved exactly via its SOCP relaxation. However, a priori it is not easy to see what system properties and conditions of the MPC problem are necessary to ensure Theorem \ref{thmSOCP} applies. The following theorem identifies these system properties and conditions.
\DeclarePairedDelimiter\set\{\}

\begin{thm} \label{thmSigmaRegion}
Consider the homogeneous MPC formulation (\ref{eq:MPC_Hom}) for controlling the discrete linear system (\ref{eq:discrete_dynamics}) over a horizon length $N$. Define $n = Nn_u$ as the dimension of the decision variable $\Uo$. Let the system dynamics $(A,B)$, cost and constraint matrices ($Q_i, R_i, S_i, i = 0,\hdots,m$) be such that the family of matrices $M_i \in \mathbb{S}^n (i = 0,\hdots,m)$ defined by (\ref{eq:Mi}) is uniformly almost off-diagonal non-positive with respect to a given vector $\sigma \in \{-1,+1\}^n$. Then (\ref{eq:MPC_Hom}) can be solved exactly using the SOCP relaxation (\ref{eq:SOCP}) and reconstructing $\Uo$ according to (\ref{eq:z}) with $\bar{\sigma}^+ = \begin{bmatrix}
1 && \sigma_1 && \hdots && \sigma_n 
\end{bmatrix}^T$ when $x_0 \in \mathbb{X}^{+}$ and $\bar{\sigma}^- = \begin{bmatrix}
1 && -\sigma_1 && \hdots && -\sigma_n 
\end{bmatrix}^T$ when $x_0 \in \mathbb{X}^{-}$ where $\mathbb{X}^{+}$ and $\mathbb{X}^{-}$ are given by:
\begin{align} \label{eq:sigmaRegion}
\mathbb{X}^{+} = \set{x \ | \ [x_0^TN_i + d_i^T]_{1k} \sigma_k \leq 0}, \\
\mathbb{X}^{-} = \set{x \ | \ [x_0^TN_i + d_i^T]_{1k} \sigma_k \geq 0}, \\
0 \leq i \leq m, 1 \leq k \leq n\ \nonumber
\end{align} 
\end{thm}

\begin{proof}
By Definition~\ref{def:odnp_family}, the family of matrices $P_i(x_0), i = 0,\hdots,m,$ is uniformly almost off-diagonal non-positive with respect to $\bar{\sigma}^+$ if:
\begin{align}
[P_i(x_0)]_{jk}\bar{\sigma}^+_j \bar{\sigma}^+_k \leq 0& \label{eq:ineq0} \\
0 \leq i \leq m, 1 \leq j < k \leq n+1& \nonumber
\end{align}
Given that $\bar\sigma^+_1 = 1$, it is straight-forward to see that this is equivalent to the conditions
\begin{align}
[x_0^TN_i + d_i^T]_{1k} \sigma_k \leq 0& \label{eq:ineq1}\\
[M_i]_{jk}\sigma_j \sigma_k \leq 0&  \label{eq:ineq2} \\
0 \leq i \leq m, 1 \leq j < k \leq n& \nonumber
\end{align}
Inequality (\ref{eq:ineq2}) is satisfied by the stated assumption that $M_i$ is uniformly almost off-diagonal non-positive with respect to $\sigma$. Thus we have (\ref{eq:ineq1})$\iff$(\ref{eq:ineq0}). Let $\mathbb{X}^+$ denote the set of vectors that satisfy (\ref{eq:ineq1}). When $x_0 \in \mathbb{X}^+$, the family of matrices $P_i(x_0)$ is uniformly almost off-diagonal non-positive with respect to $\bar\sigma^+$ and Theorem \ref{thmSOCP} applies. A nearly identical proof establishes that $P_i(x_0)$ is uniformly almost off-diagonal non-positive with respect to $\bar\sigma^-$ for $x_0 \in \mathbb{X}^-$. 
\end{proof}

\begin{remark}
Theorem \ref{thmSigmaRegion} allows us to characterize a class of non-convex MPC problems that can be solved using the SOCP relaxation of Theorem \ref{thmSOCP}. Notably, the solvability of the problem depends on the initial condition $x_0$ via the sets $\mathbb{X}^+$ and $\mathbb{X}^-$.

It is possible that, for different initial conditions, the conditions of Theorem \ref{thmSigmaRegion} are satisfied for different $\sigma$. 
The follow lemma further illustrates the relationship between $P_i(x_0)$, $\sigma$ and the corresponding sets $\mathbb{X}^+$ and $\mathbb{X}^-$.
\end{remark}
\begin{lem} Given a single matrix $P_i(x_0) \in \mathbb{S}^{n+1}, i \in \mathbb{N}$ that is almost off-diagonal non-positive with respect to some $\bar\sigma \in \{-1,+1\}^{n+1}$ then $-P_i(x_0)$ is also almost off-diagonal non-positive with respect to $\bar\sigma$ if and only if it is diagonal.\label{lemDiag}
\end{lem} 
\begin{proof}
Sufficient: $P_i(x_0)$ is diagonal implies $[P_i(x_0)]_{jk} = 0 \ (1 \leq j < k \leq n+1)$. Applying Definition~\ref{def:odnp}, a matrix is almost off-diagonal non-positive with respect to $\bar\sigma$
if $[P_i(x_0)]_{jk}\bar\sigma_j\bar\sigma_k \leq 0 \ (1 \leq j < k \leq n+1)$. Given a diagonal matrix, this relationship is true for arbitrary $\bar\sigma$. If $P_i(x_0)$ is diagonal then $-P_i(x_0)$ is also diagonal and therefore almost off-diagonal non-positive with respect to any $\bar\sigma$. Necessary: Consider a matrix $P_i(x_0) \in \mathbb{S}^{n+1} $ with element $[P_i(x_0)]_{jk} \neq 0, (j < k)$ that is almost off-diagonal non-positive with respect to $\bar\sigma$. This implies $[P_i(x_0)]_{jk}\bar\sigma_j\bar\sigma_k < 0$ and thus $[-P_i(x_0)]_{jk}\bar\sigma_j\bar\sigma_k > 0$. Therefore $-P_i(x_0)$ cannot be almost off-diagonal non-positive with respect to $\bar\sigma$.
\end{proof}
\begin{remark}[] Diagonal $P_i(x_0)$ includes the important case of norm bounds on the control vector as given by 
$l_b \leq \Uo^T R \Uo \leq u_b$ where $R$ is diagonal with non-negative entries and $l_b, u_b \in \R$ are the lower and upper bounds respectively. When $l_b > 0$ and $R$ contains more than one non-zero entry, the resulting constraint is non-convex. This can be rewritten as two constraints $-\Uo^T R \Uo \leq l_b$ and $\Uo^T R \Uo \leq u_b$ both of which give diagonal matrices when put in the form (\ref{eq:Pix0}). We note that non-convex control constraints of this form arise in thrust vectoring problems \cite{Aci2013}.
\end{remark}

\begin{remark}[]
Linear state weightings of the form $c^T\X$ with $c \in \R^{(N+1)n_x}$ translate to off-diagonal entries in (\ref{eq:Pix0}). If a lower bound $l_b$ and upper bound $u_b$ is applied to a given state weighting, one obtains two equal and opposite matrices $P_i(x_0)$ and $-P_i(x_0)$ with off-diagonal terms. Applying Lemma \ref{lemDiag}, both of the matrices cannot be almost off-diagonal non-positive with respect to a given $\bar\sigma$. Thus Theorem \ref{thmSigmaRegion} does not support MPC formulations with lower and upper bounds applied to a given state weighting. In adversarial control this is not a major limitation in practice as one is not attempting to keep the system state within some prescribed bounds.
\end{remark}

\subsection{Adversarial Control of Positive Systems}
The previous section established state-dependent conditions under which the homogeneous adversarial MPC formulation (\ref{eq:MPC_Hom}) can be solved by applying Theorem \ref{thmSOCP}. By restricting ourselves to positive systems, we establish conditions under which Theorem \ref{thmSOCP} holds for all $x_0 \geq 0$ (i.e. the positive orthant).
\begin{thm} \label{thm:Positive}
Consider the homogeneous MPC formulation (\ref{eq:MPC_Hom}) for controlling a discrete-time linear system (\ref{eq:discrete_dynamics}) over a horizon length $N$. Let the system dynamics $(A,B)$ be positive as described in Definition \ref{def:positiveSystem}. Define $n = Nn_u$ as the dimension of the decision variable $\Uo$.  Let the cost and constraint matrices be such that $Q_i \leq 0, \ [R_i]_{jk} \leq 0 \ (j \neq k), \ S_i \leq 0, \ q_i \leq 0, \ r_i \leq 0$ for $ i = 0,\hdots,m$. Then (\ref{eq:MPC_Hom}) can be solved using Theorem \ref{thmSOCP} with $\bar\sigma^+ = \mathds{1}_{n+1}$ when $x_0 \geq 0$.
\end{thm} 
\begin{proof}
The proof is simple but involves some tedious algebra. For clarity, we outline the main steps below:
\begin{enumerate}
  \item Show that $[M_i]_{jk} \leq 0, \ \ 1 \leq j < k \leq n,\ 0 \leq i \leq m$. \newline
  \textit{Proof}: See below
  \item Show that $N_i \leq 0, d_i \leq 0, \ \ 0 \leq i \leq m$\newline
  \textit{Proof}: See below
  \item $N_i \leq 0, d_i \leq 0, x_0 \geq 0 \implies (x_0^TN_i + d_i^T) \leq 0$
  \item Steps 1 and 3 imply $[P_i(x_0)]_{jk} \leq 0 \ \forall \ (j \neq k, x_0 \geq 0)$. Therefore $P_i(x_0)$ is uniformly almost off-diagonal non-positive with respect to $\bar\sigma^+ = \mathds{1}_{n+1}$ and (\ref{eq:MPC_Hom}) can be solved using Theorem \ref{thmSOCP}.
\end{enumerate}
\textit{Step 1)} Recall that the product of two non-negative matrices is itself non-negative. We are given that $A \geq 0, B \geq 0$. By induction, the products $A^i \geq 0, A^iB \geq 0,~\forall~i \in \mathbb{N}$. This implies $\Sx \geq 0$ and $\Su \geq 0$ as all the individual non-zero entries shown in (\ref{eq:XUSxSu}) can be written in terms of $A^i$ and $A^iB$ for some $i \in \mathbb{N}$. Recall that the product of a non-negative matrix and non-positive matrix is non-positive.  So $S_i \leq 0$, $S_u \geq 0, Q_i \leq 0 \implies S_u^TS_i \leq 0, \ S_u^TQ_iS_u \leq 0$ and therefore $S_u^TQ_iS_u + S_u^TS_i + S_i^TS_u \leq 0$. Lastly, we are given that $[R_i]_{jk} \leq 0 \ (j \neq k)$. From (\ref{eq:Mi}), $M_i$ = $S_u^TQ_iS_u + S_u^TS_i + S_i^TS_u + R_i$. Combining the previous results establishes that $[M_i]_{jk} \leq 0 \ (1 \leq j < k \leq n, 0 \leq i \leq m)$. \newline \newline
\textit{Step 2)} Given $S_x \geq 0$, $S_u \geq 0$, $Q_i \leq 0$, $S_i \leq 0$, $q_i \leq 0$, and $r_i \leq 0$, similar reasoning as Step 1 establishes that $N_i \leq 0$ and $d_i \leq 0$ as defined by (\ref{eq:Ni}) and (\ref{eq:di}) respectively.
\end{proof}
\begin{remark}
As $\bar\sigma^+= \mathds{1}_{n+1}$ determines the sign pattern of the solution, the resulting control sequence $\Uo$ is non-negative. A positive system will remain in the positive orthant under the action of this control sequence per Lemma \ref{lem:positiveSystem}. 
\end{remark}
\begin{remark}
Theorem~\ref{thm:Positive} includes the practical case of an objective function with diagonal $Q_0 < 0$ and diagonal $R_0 > 0$. This represents a situation in which an adversary is attempting to push the system away from the origin while minimizing the energy expended to do so.
\end{remark}

\section{Numerical Examples}
We demonstrate our results on some simple systems. To clearly point out sources of non-convexity, we write the examples in uncondensed form with state variables appearing in the cost function. However, the resulting problems are solved by converting the problem to the form of (\ref{eq:MPC_Hom}) and applying Theorem \ref{thmSOCP}.
 
\subsection{Indefinite Cost Function}
Our first example applies an indefinite cost function to a two-state system. This allows us to show graphically the regions $\mathbb{X}^+$ and $\mathbb{X}^-$ where we can solve the problem exactly. Consider the following discrete state-space model: 
\begin{equation*}
A = \begin{bmatrix}
0.9 && -0.2 \\
0 && 0.9 \\
\end{bmatrix}
\quad
B = \begin{bmatrix}
0.2 && -0.05 \\
0 && 2 \\
\end{bmatrix}
\end{equation*}
We apply an indefinite quadratic objective of minimizing the product of the two states over a horizon $N$. The control at each step $k$ is constrained to an annulus in $\R^2$. Additionally, the total control effort over the horizon $N$ is constrained, reflecting energy constraints. 
\begin{mini*}|s| 
{}{\sum_{k=0}^{N} x(k)^T\begin{bmatrix} 0 && 1 \\ 1 && 0 \end{bmatrix}{x(k)}}
{}{}
\addConstraint{\X}{= \Sx x(0) + \Su \Uo}
\addConstraint{0.2}{\leq \norm{u(k)}_2^2 \leq 0.5,}{\quad k = 0,\hdots,N-1}
\addConstraint{0}{\leq \norm{\Uo}_2^2 \leq \frac{N}{3}}
\label{eq:MPCEx1}
\end{mini*} 
We rewrite this in condensed form. Dropping constant terms, the resulting cost function becomes:
\begin{equation*}
G_0(x(0),\Uo) = \\
\Uo^T M_0 \Uo + 2(x(0)^TN_0)\Uo
\end{equation*}
Where for $N = 2$ we have:
\begin{equation*}
M_0 = 
\begin{bmatrix}
       0 &&   0.0724   &&       0 &&    0.0360 \\
    0.0724 &&   -0.0506  &&   0.0360 &&    -0.0260 \\
         0  &&   0.0360   &&        0  &&   0.0400 \\
    0.0360 &&   -0.0260 &&    0.0400  &&  -0.0200 
\end{bmatrix}
\end{equation*}

\begin{equation*}
N_0 = \begin{bmatrix}
         0  &&  0.3258   &&      0 &&   0.1620 \\
    0.3258   && -0.2187  &&  0.1620  &&  -0.1125
\end{bmatrix}
\end{equation*}

Here there is no offset term $d_0$ as we have no linear terms $(q_0, r_0)$ in our original, uncondensed cost. $M_0$ is off-diagonal non-positive with respect to $\sigma = \begin{bmatrix}
1 && -1 && 1 && -1
\end{bmatrix}$.
From Theorem \ref{thmSigmaRegion} the SOCP relaxation is exact for $x(0) \in \mathbb{X}^+ \cup \mathbb{X}^-$ where $\mathbb{X}^+ = \set{x \ | \ x \in \R^2, [x(0)^TN_0]_{1k} \sigma_k \leq 0 \ (1 \leq k \leq 4)}$ and $\mathbb{X}^-$ is similarly defined. Although $\mathbb{X}^+$ and $\mathbb{X}^-$ are described by the intersection of four hyperplanes which pass through the origin, we can limit ourselves to the two hyperplanes whose normal vector has the smallest inner product. This gives the following:
\begin{align*}
&\mathbb{X}^+ = \{x \ | \ x \in \R^2, x_2 \leq 0, -0.32358x_1 + 0.2187x_2 \leq 0\} \\
&\mathbb{X}^- = \{x \ | \ x \in \R^2, x_2 \geq 0, -0.32358x_1 + 0.2187x_2 \geq 0\}
\end{align*}

Figure \ref{fig:sigmaplot} shows the regions $\mathbb{X}^+$, $\mathbb{X}^-$ when $N = 2$. A sample trajectory is shown starting from $x(0) = [0 \ 0.1]^T$. With a horizon of $N = 2$ we only obtain control commands $u(0)$ and $u(1)$. As is standard in MPC, we apply the first command $u(0)$ which takes us to state $x(1)$. Redefining $x(1)$ as our new initial condition we then resolve the problem.  We repeat this process $10$ times to obtain the trajectory shown.   
\begin{figure}[htp]
    \centering
    \includegraphics[width=0.5\textwidth]{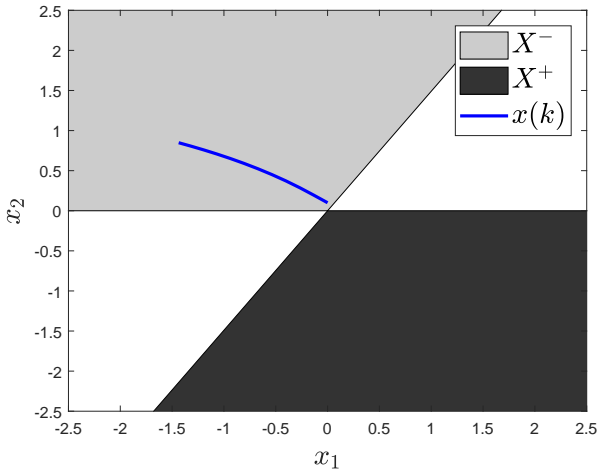}
    \caption{Indefinite MPC example ($N = 2$) starting at $x(0) = [0 \ 0.1]^T$}
    \label{fig:sigmaplot}
\end{figure}
\begin{remark}
$\mathbb{X}^+$ and $\mathbb{X}^-$ are described by the intersection of halfspaces formed from the columns of $N_0 \in \R^{n_x \times Nn_u}$. Interestingly for this problem, as $N$ is increased the sets $\mathbb{X}^+$, $\mathbb{X}^-$ cover a larger portion of $\R^2$. For example, Figure \ref{fig:sigmaplot2} shows $\mathbb{X}^+$, $\mathbb{X}^-$ for $N = 20$. With this horizon length we can solve a trajectory starting at $x(0) = [1 \ 0.5]^T$ which is outside the solvable regions when $N = 2$.
\end{remark}
\begin{figure}[htp]
    \centering
    \includegraphics[width=0.5\textwidth]{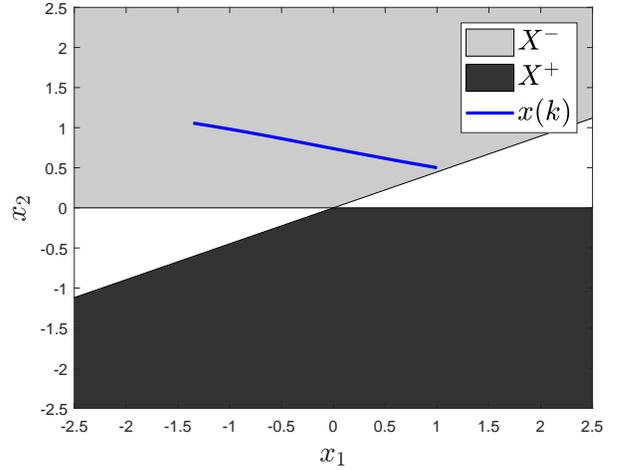}
    \caption{Indefinite MPC example ($N = 20$) starting at $x(0) = [1 \ 0.5]^T$}
    \label{fig:sigmaplot2}
\end{figure}
\begin{remark} Minimizing or maximizing the product of two states is frequently seen in economic MPC formulations. In some instances, an indefinite stage cost can still yield a convex problem if applied over a sufficiently long horizon \cite{Ber2018}. That does not occur here. Instead the cost function remains indefinite with $N$ positive eigenvalues and $N$ negative eigenvalues for a given horizon length $N$.
\end{remark}
\subsection{Adversarial Control of Double Integrator}
Consider a simple planar double integrator model of an autonomous vehicle with position states $(p_x, p_y)$ and associated velocity states $(v_x, v_y)$. State feedback damping terms regulate the system to the origin. An adversary is able to apply disturbance forces $(u_1, u_2)$ to the system. The continuous dynamics are given by:
\begin{align*} \label{eq:DoubleInt}
\frac{dp_x}{dt} &= -0.1p_x + v_x & \frac{dv_x}{dt} &= -0.1v_x + u_1 \\
\frac{dp_y}{dt} &= -0.1p_y + v_y & \frac{dv_y}{dt} &= -0.1v_y + u_2
\end{align*}
We discretize the continuous model using a zero-order-hold with 0.2s sample time obtaining matrices $(A,B)$ with state vector $x = [p_x \ p_y \ v_x \ v_y]^T$ and control $u = [u_1 \ u_2]^T$. By inspection the discrete model is positive.
\begin{equation*}
\begin{gathered}
A = \begin{bmatrix}
0.9802 && 0 && 0.196 && 0 \\
0 && 0.9802 && 0 && 0.196 \\
0 && 0 && 0.9802 && 0\\
0 && 0 && 0 && 0.9802
\end{bmatrix} \\
B = \begin{bmatrix}
0.01974 && 0 \\
0 && 0.01974 \\
0.198 && 0 \\
0 && 0.198
\end{bmatrix}
\end{gathered}
\end{equation*}
We are given a safety envelope defined by the union of two ellipsoids centered at the origin. The adversaries objective is ensure the system's position is outside this safe operating envelope by the end of a horizon $N = 10$ while minimizing energy expenditure. This terminal position constraint is non-convex. The available control magnitude is bounded to be within an annulus representative of thrust vectoring constraints. The resulting adversarial MPC problem is:
\begin{mini*}|s| 
{\Uo}{\norm{\Uo}}{}{}
\addConstraint{1.0}{ \leq \Big( \frac{p_x(k)}{1.0} \Big)^2 + \Big( \frac{p_y(k)}{0.5}\Big)^2 ,\, \,  k = N}
\addConstraint{1.0}{ \leq \Big( \frac{p_x(k)}{0.5} \Big)^2 + \Big( \frac{p_y(k)}{1.0}\Big)^2 ,\, \,  k = N}
\addConstraint{0.04}{ \leq u_1^2(k) + u_2^2(k) \leq 0.25, \quad k = 0,\hdots,N-1}
\end{mini*} 
Written in the standard quadratic form of (\ref{eq:F}), the terminal position constraints have matrices $Q_i \leq 0$ while the control constraints consist of diagonal $R_i$. Thus Theorem \ref{thm:Positive} applies and we can solve this non-convex problem when $x_0 \geq 0$. Figure \ref{fig:stayout} plots the ellipse bounds in the positive orthant and shows sample trajectories with varying initial positions. Figure \ref{fig:stayout_control} shows the associated control command. The initial velocities are zero in each example. Starting at point $(0,0)$ the adversarial control pushes the system towards the closest point on the border of the safety envelope, reaching this point only at the end in order to minimize the energy expended. Starting at $(0.5,0)$ the damped dynamics of $A$ are evident as the trajectory initially moves towards the origin. Finally, starting closer to the boundary at $(0.1,0.7)$, the trajectory overshoots the boundary. This is due to the non-convex lower bound on the control magnitude which prevents us from turning off the control.
\begin{figure}[h]
    \centering
    \includegraphics[width=0.5\textwidth]{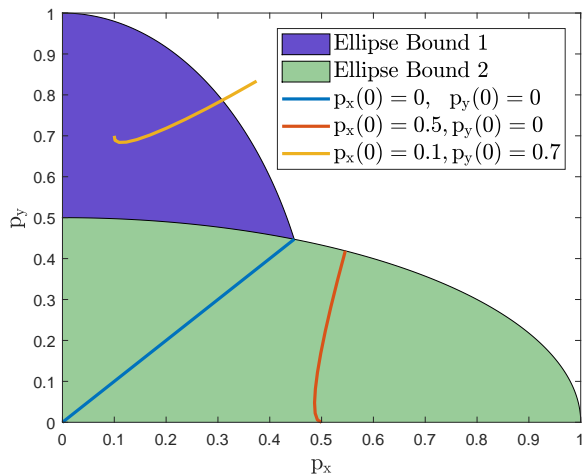}
    \caption{Safety envelope violation with minimum energy expenditure}
    \label{fig:stayout}
\end{figure}
\begin{figure}[h]
    \centering
    \includegraphics[width=0.5\textwidth]{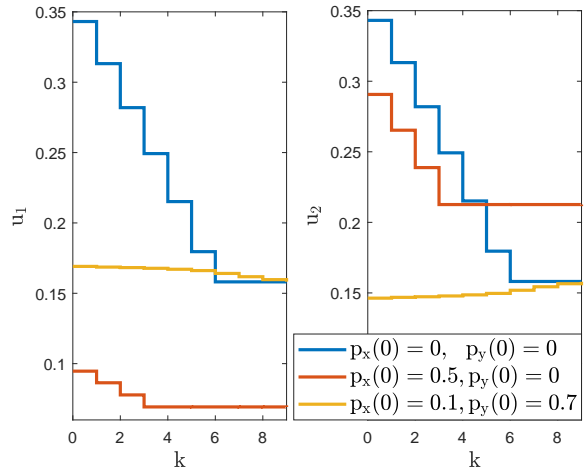}
    \caption{Control history for safety envelope violation}
    \label{fig:stayout_control}
\end{figure}
\subsection{Maximizing Voltage Mismatch within a Microgrid}
Finally we consider a simple microgrid model consisting of three buses. Without loss of generality, the origin is taken to be the equilibrium point. Each bus $i$ is modeled as a capacitor $c_i$ with voltage $v_i$. The buses are interconnected by resistive transmission lines $r_2$ and $r_3$. Collectively they supply power to a resistive load $r_1$ and a constant power load whose linearized dynamics can be represented by a negative resistance $r_4$. An adversary is able to inject current into the system through $i_1$ and $i_2$. Table \ref{table:1} lists the parameters. 
\begin{table}[h]
\caption{Microgrid Parameters}
\begin{center}
 \begin{tabular}{||c c c c c c c||} 
 \hline
 $c_1$ & $c_2$ & $c_3$ & $r_1$ & $r_2$ & $r_3$ & $r_4$\\ [0.5ex] 
 \hline
 0.2 & 0.2 & 0.2 & 8 & 1 & 0.5 & -10 \\
 \hline
\end{tabular}
\end{center}
\label{table:1}
\end{table}
\\The continuous dynamics are given by:
\begin{align*} \label{eq:PowerSystem}
c_1\frac{dv_1}{dt} &= -\frac{1}{r_1}v_1 - \frac{1}{r_2}(v_1 - v_2) + i_1 \\
c_2\frac{dv_2}{dt} &= -\frac{1}{r_2}(v_2 - v_1) - \frac{1}{r_3}(v_2 - v_3) \\
c_3\frac{dv_3}{dt} &= -\frac{1}{r_3}(v_3 - v_2) - \frac{1}{r_4}v_3 + i_2
\end{align*}
The discrete model with time-step 0.1s is:
\begin{equation*}
A = \begin{bmatrix}
0.6282 && 0.2221 && 0.1026 \\
0.2221 && 0.4171 && 0.3646 \\
0.1026 && 0.3646 && 0.5663
\end{bmatrix}
\end{equation*}
\begin{equation*}
B = \begin{bmatrix}
0.3941 && 0.0213 \\
0.0716 && 0.1266 \\
0.0213 && 0.3616
\end{bmatrix}
\end{equation*}
with state vector $x = [v_1 \ v_2 \ v_3]^T$ and control $u = [i_1 \ i_2]^T$. By inspection the discrete model is positive.

In traditional microgrid voltage regulation, the controls would attempt to achieve consensus on the voltages $(v_1 = v_2 = v_3$). Here we focus on maximizing disagreement by injecting currents $i_1$ and $i_2$. The voltage disagreement at time index $k$ is defined as:
\begin{equation*}
J(k) = (v_1(k) - v_2(k))^2 + (v_1(k) - v_3(k))^2 + (v_2(k) - v_3(k))^2
\end{equation*}
We use a horizon length of $N = 20$ and maximize disagreement at the end.
\begin{mini*}|s| 
{\Uo}{-J(N)}
{}{}
\addConstraint{i_1^2(k) + i_2^2(k)}{\leq 1,}{\qquad k = 0,\hdots,N-1}
\label{eq:MPCPower} 
\end{mini*} 
\setcounter{MaxMatrixCols}{40}
The resulting condensed MPC formulation has two negative eigenvalues, with the rest zero. Although the system is positive, the matrix of the quadratic cost function $-J(N)$ contains positive off-diagonal terms and thus we cannot apply Theorem \ref{thm:Positive}. A simple numerical check reveals that the problem is uniformly almost off-diagonal non-positive with respect to $\sigma = [\begin{smallmatrix}
1 && -\mathds{1}^T_{30} && 1 && -1 && 1 && -1 && 1 && -1 && 1 && -1 && 1 && -1
\end{smallmatrix}]$.

Figure \ref{fig:voltagecurrent} shows the resulting state and control trajectory with all states initially zero. At the end, the disagreement in voltages is maximized. As $\sigma$ contains both $+1$ and $-1$ entries the resulting control sequences $i_1(k)$ and $i_2(k)$ contain both positive and negative terms.

\begin{figure}[h]
    \centering
    \includegraphics[width=0.5\textwidth]{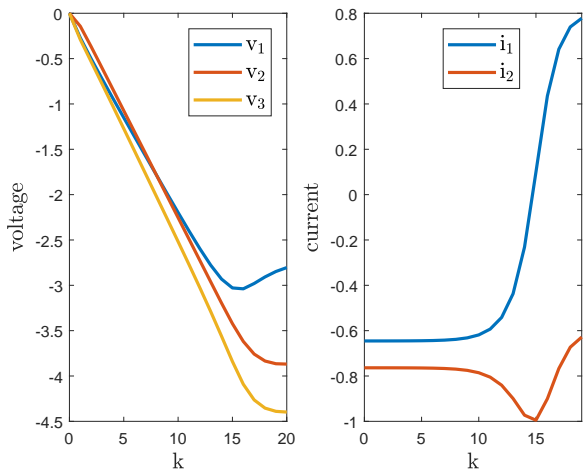}
    \caption{Maximizing voltage disagreement in a microgrid}
    \label{fig:voltagecurrent}
\end{figure}

\subsection{Implementation Details}
All examples were solved using MOSEK \cite{Mos2017} in conjunction with YALMIP \cite{Lof2004}. For sufficiently small problems we also solved the original non-convex QCQP using the global optimization solver \textit{BMIBNB} in YALMIP. This solver implements a simple branch-and-bound algorithm which can find global solutions to arbitrary optimization problems of modest size. In all instances, the solution obtained matched that provided by the SOCP formulation. Although our focus is not on solver efficiency, we note that for a problem with 20 decision variables the SOCP formulation was consistently solved in under $50ms$ while solving with \textit{BMIBNB} took over 100 seconds. Larger problems were not validated with \textit{BMIBNB} due to excessive runtimes. 

\section{CONCLUSIONS}
In this work we established conditions under which non-convex, adversarial model predictive control problems can be solved to global optimality via second-order cone programming. For general systems, the global solution can only be obtained in a subspace of the whole state-space. It was shown that many adversarial problems are readily solved for systems whose dynamics are invariant with respect to the positive orthant. Future work will examine whether similar conditions can be identified for systems which exhibit other forms of invariance. For cases in which the system does not admit an exact SOCP solution, we plan to combine our methods with heuristics for approximately solving the resulting indefinite QCQP \cite{Par2017}. 
 
\addtolength{\textheight}{-12cm}   

\bibliography{myref}
\bibliographystyle{ieeetr}

\end{document}